\documentclass[11pt]{paper}
\usepackage[scale=.7]{geometry}

\usepackage {amssymb}
\usepackage {amsmath}
\usepackage {amsthm}
\usepackage [usenames] {color}

\usepackage {enumerate}

\usepackage {cite}
\usepackage{float}

\usepackage {xspace}
\usepackage {wrapfig}

\definecolor {infocolor} {rgb} {0.6,0.6,0.6}

\floatstyle{ruled}
\newfloat{algorithm}{thp}{alg}
\floatname{algorithm}{Algorithm}

\newcommand {\mathset} [1] {\ensuremath {\mathbb {#1}}}
\newcommand {\R} {\mathset {R}}

\usepackage{graphicx}
\usepackage{setspace}
\usepackage{bbold, hyperref}

\usepackage{tikz,pgfplots}
\usetikzlibrary{arrows,matrix,positioning}
\usetikzlibrary{calc,fadings,decorations.pathreplacing}
\newtheorem{Definition}{Definition}

\newtheorem{Proposition}{Proposition}
\newtheorem{Theorem}{Theorem}
\newtheorem{Lemma}{Lemma}
\newtheorem{Corollary}{Corollary}

\newcommand{\bsm}{\left[\begin{matrix}}
\newcommand{\esm}{\end{matrix}\right]}

\tikzset{%
  >=latex, 
  inner sep=0pt,%
  outer sep=2pt,%
  mark coordinate/.style={inner sep=0pt,outer sep=0pt,minimum size=3pt,
    fill=black,circle}%
}

\makeatletter
\newsavebox{\sfe@box}
{\color@endgroup\egroup\subfloat[\sfe@caption]%
{\usebox{\sfe@box}}}
\makeatother

\begin{document}



\title{Sparse Stable Matrices}


\author{M.-A. Belabbas\thanks{\href{mailto: belabbas@illinois.edu}{belabbas@illinois.edu}}\\
{\small Coordinated Science Laboratory}\\\small{University of Illinois at Urbana-Champaign}
}

\maketitle

\begin{abstract}
In the design of decentralized networked systems, it is useful to know whether a given network topology can sustain stable dynamics. We consider a basic version of this problem here: given  a vector space of sparse real  matrices, does it  contain a stable (Hurwitz) matrix? Said differently, is a feedback channel (corresponding to a non-zero entry)  necessary for stabilization or can it be done without. We provide in this paper a set of necessary and a set of sufficient conditions for the existence of  stable matrices in a vector space of sparse matrices. We further prove some properties of  the set of sparse  matrix spaces that contain Hurwitz matrices. The conditions we exhibit are most easily stated in the language of graph theory, which we thus adopt in this paper. 

\end{abstract}

\begin{small}
\noindent \emph{Keywords}: Stability; Decentralized Control; Network Control; Graph Theory; Hamiltonian cycles
\end{small}



\section{Introduction}
Many problems of practical and theoretical nature in control, biology and  communications  are characterized by an underlying network topology describing which interactions within a system are allowed, see e.g.~\cite{may_complex_stab_nature_72, rotko_decentr_06,  fax_information_2004, ali_jadbabaie_coordination_2003,  shreyas_sundaram_distributed_2011, a._nedic_distributed_2009,rinehart_networked_11} and references therein. Such problems include information transmission, distributed computation,  the study of metabolic networks, robot motion planning~\cite{desai2001modeling}, etc. The solution to particular problems of this nature can often be obtained from the broader study of the class of systems constrained by the given network topology. In this vein, the problem we address here is the one of stability of linear dynamics. The question we answer is: ``can  a given network topology  sustain stable linear dynamics?" We give a set of necessary and a set of sufficient conditions for the existence of stable dynamical systems constrained by a given topology. Furthermore, we establish some structural properties of the set of topologies that can sustain stable dynamics.

The  stability of linear dynamics is  well-understood  when there are no communication constraints in the system. Necessary and sufficient conditions in this case are given by the well-known  Routh-Hurwitz criterion~\cite{horn_matrix_1990}. These conditions are, however, not easily managed and their application in the networked case does not appear to yield manageable results. We adopt a different approach in this paper: we first relate the algebraic representation of a class of systems sharing the same underlying topology to a graphical representation and  exhibit continuous and discrete symmetries that leave the class of systems invariant; from there,  we define an algebraic set whose symmetries make it easier to derive conditions for stability. While the use of graphical methods in control is not new (see~\cite{lin74,egerstedtbook2010} or the early monograph~\cite{ Reinschke88}), our approach departs from earlier work in the sense that it relates graph theoretic properties to systems' stability.  We will show in this paper that Hamiltonian decompositions of graphs are a natural ingredient if one is to relate systems' stability to an underlying graph. Indeed, such decompositions appear in the statement of several of our main results below (Theorem~\ref{th:mainthnec}, Theorem~\ref{th:mainthsuff} and Theorem~\ref{th:structV}). Another important notion that is introduced in this work is the one of minimally stable graphs (defined in the next section); in fact,  one objective of the present line of work is to characterize all minimally stable graphs. We make a first step in this direction in the last section.

The paper is structured as follows. In the next section, we give a general overview of the problem, establish the notation and state the main results. In Section~\ref{sec:groupactions}, we introduce the group actions and the algebraic set just mentioned. In Sections~\ref{sec:hamildec} and~\ref{sec:setv}, we prove the main results of the paper.

\section{Preliminaries and summary of results}\label{sec:prelim}

Let $n>0$ be a positive integer. Denote by $\R^{n \times n}$ the vector space of real $n$ by $n$ matrices. We are concerned with some linear subspaces of $\R^{n \times n}$, namely the ones obtained by forcing one or several entries of the matrices to be zero. Precisely, let $\alpha$ be a set of pairs of integers between $1$ and $n$, that is $\alpha \subset \lbrace 1, \ldots, n \rbrace \times \lbrace 1,\ldots, n \rbrace$. We define $\Sigma_\alpha$ to be the vector space of matrices with entries indexed by $\alpha$ free, and other entries set to zero. We denote by $|\alpha|$ the cardinality of $\alpha$. For example, if $n=4$ and $\alpha=\lbrace (1,2),(1,3),(2,1),(2,2),(2,3),(3,2),(4,1),(4,3),(4,4) \rbrace$, then $\Sigma_\alpha$ is the subspace of matrices of the form
\begin{equation}
\label{ex:ma}A = \left[\begin{matrix}
0 & \ast & \ast & 0\\
\ast & \ast & 0 &\ast \\
0 & \ast & 0 & 0\\
\ast &0 & \ast & \ast
\end{matrix}\right]
\end{equation}
 where the $\ast$ represent arbitrary real numbers. We call the vector spaces $\Sigma_\alpha$  \emph{sparse matrix spaces} (SMS).  We call the $\ast$ entries \emph{free variables} and the other entries \emph{zero variables}.

We say that a matrix is stable if the real-part of its eigenvalues is strictly negative. We call a \emph{sparse matrix space stable} if it contains a stable matrix, and unstable otherwise. Our objective is to characterize the stable and unstable sparse matrix spaces.  To this end, it is convenient to adopt the following graphical representation for the $\Sigma_\alpha$'s:
to a vector subspace $\Sigma_\alpha \subset \R^{n \times n}$, we associate a directed graph (digraph) $G=(V,E)$ with $n$ vertices and $\dim \Sigma_\alpha = |\alpha |$ edges defined as follows: there is an edge between vertices $v_i$ and $v_j$ if   $(i,j)$ is  in $\alpha$, or, said otherwise, if $a_{ij}$ is a free variable. Hence the more sparse the matrices in $\Sigma_\alpha$, the fewer edges in  the corresponding graph.

\begin{figure}
\begin{center}

\begin{tikzpicture}[scale = .125, ->,>=stealth,shorten >=1pt,auto,node distance=2cm,
  thin,main node/.style={circle,fill=blue!20,draw}]

  \node[main node] (1) {1};
  \node[main node] (2) [below left of=1] {2};
  \node[main node] (3) [below right of=2] {3};
  \node[main node] (4) [below right of=1] {4};

  \path[every node/.style={font=\sffamily\small}]
    (1)      edge [bend right=10]  (2)
   edge   (3)

    (2) edge [bend right=10](1)
        edge (4)
        edge [loop left]  (2)
    (3) edge (2)
       
    (4) edge  (3)
    	edge (1)
     edge [loop right] (4)       ;
\end{tikzpicture}
\end{center}
\caption{The graph depicts the vector space $\Sigma_\alpha$ described in~\eqref{ex:ma}}
\label{fig:ma2}
\end{figure}
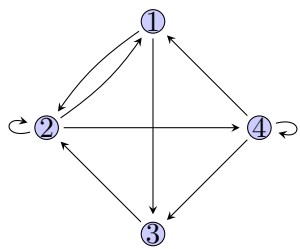

Some rather intuitive necessary conditions for stability are easily obtained and expressed in the graphical representation. The first is that  the graph has a node with a self-loop\footnote{A \emph{self-loop}, also called a \emph{buckle}, is an edge that starts and ends at the same vertex}: indeed, a stable matrix necessarily has a negative trace and since self-loops correspond to diagonal elements in the sparse matrix space, at least one free variable on the diagonal is needed. We call nodes with self-loops \emph{sinks}. We thus record the following necessary condition:
\emph{a graph is stable only if it has a node with a self-loop.}
Recall that a strongly connected component~\cite{diestel_graph_2010} is a maximal (in terms of number of nodes and edges) subgraph that is strongly connected, that is for which there exists a directed path joining any two pairs of vertices.  For example, in the graph depicted below, the subgraph on $\{1,2,3  \}$ is a strongly connected component.

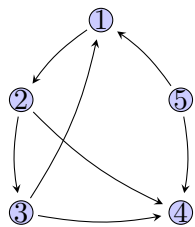
\begin{center}
\begin{tikzpicture}[scale = .25, ->,>=stealth,shorten >=1pt,auto,node distance=1.5cm,
  thin,main node/.style={circle,fill=blue!20,draw}]

  \node[main node] (1) {1};
  \node[main node] (2) [below left of=1] {2};
  \node[main node] (3) [below  of=2] {3};
  \node[main node] (5) [below right  of=1] {5};
  \node[main node] (4) [below  of=5] {4};

  \path[every node/.style={font=\sffamily\small}]
    (1)      edge [bend right=10]  (2)

    (2) edge [bend right=10](3)
    (2) edge [bend right=10](4)
        (5) edge [bend left=10](4)
edge [bend right=10](1)
    (3) edge [bend right=10](1)
   edge [bend right=10](4)
        ;
\end{tikzpicture}
\end{center}

Self-loops in a graph can be understood as elements that dissipate energy. From that point of view, it seems reasonable to expect that every node is strongly connected to a node with a self-loop, or, informally, that every node in the graph can see and is seen by an element that dissipates energy. We thus extend the previously mentioned necessary condition as follows: \emph{a graph is stable only if every node belongs to a strongly connected component with a sink.}

We summarize this discussion in the following Theorem:

\begin{Theorem}
A sparse matrix space is stable only if every node in the associated graph belongs to a  strongly connected component with a sink. \label{th:strongsink}
\end{Theorem}

This latter condition is not as straightforward to prove as the first one (hence we give a  formal proof of it below), but in some sense encompasses a condition that is fairly intuitive given the graphical representation of a SMS. This condition, however, is  far from being sufficient, a simple counter-example being given by the graph on the left in Figure~\ref{fig:ma}. One objective of this paper is, besides introducing more refined conditions, introduce the machinery and tools that allow to go beyond Theorem~\ref{th:strongsink}.
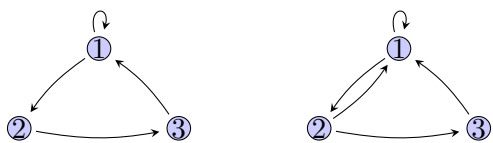
\begin{figure}
\begin{center}

\begin{tikzpicture}[scale = .25, ->,>=stealth,shorten >=1pt,auto,node distance=1.5cm,
  thin,main node/.style={circle,fill=blue!20,draw}]

  \node[main node] (1) {1};
  \node[main node] (2) [below left of=1] {2};
  \node[main node] (3) [below right of=1] {3};

  \path[every node/.style={font=\sffamily\small}]
    (1)      edge [bend right=10]  (2)
 edge [loop above] (1)   

    (2) edge [bend right=10](3)
    (3) edge [bend right=10](1)
   
        ;
\end{tikzpicture}\qquad\qquad
\begin{tikzpicture}[scale = .25, ->,>=stealth,shorten >=1pt,auto,node distance=1.5cm,
  thin,main node/.style={circle,fill=blue!20,draw}]

  \node[main node] (1) {1};
  \node[main node] (2) [below left of=1] {2};
  \node[main node] (3) [below right of=1] {3};

  \path[every node/.style={font=\sffamily\small}]
    (1)      edge [bend right=10]  (2)
 edge [loop above] (1)   

    (2) edge [bend right=10](3)
    (2) edge [bend right=10](1)
    
    (3) edge [bend right=10](1)
   
        ;
\end{tikzpicture}
\end{center}
\caption{ Even though both graphs are strongly connected, the SMS with corresponding graph depicted on the left is not stable, whereas the one with corresponding graph depicted on the right is. We provide in this paper (see Theorem~\ref{th:mainthnec} and~\ref{th:mainthsuff}) conditions that are strong enough to explain why the behavior of systems defined of these two graphs differ.}
\label{fig:ma}
\end{figure}


In addition to necessary and sufficient conditions for stability, we  investigate structural properties of the \emph{set of sparse matrix spaces} with an eye towards providing a classification of stable and unstable topologies. To understand the nature of these results, it is helpful to keep in mind the rule of thumb stating that adding edges helps stability and removing edges helps instability. This leads to the definition of \emph{minimally stable graphs}: graphs that are stable and such that removing any edge yields an unstable graph. To be more formal,   observe that sparse matrix spaces can naturally be partially ordered according to subspace inclusion and that if a sparse matrix space $\Sigma_\alpha$ is unstable, every subspace of $\Sigma_\alpha$ unstable. Reciprocally, if $\Sigma_\beta$ is stable, every subspace that contains $\Sigma_\beta$ is stable. In terms of the graphical representation, graphs partially ordered as just described are easily seen to correspond to graphs with added or removed edges.   The above facts allow us to talk about \emph{minimally stable sparse matrix spaces}, that is stable sparse matrix spaces  such that every sparse matrix space obtained from it by adding a zero (viz. removing an edge)  is unstable, and \emph{maximally unstable sparse matrix spaces},  that is unstable sparse matrix spaces such that every sparse matrix space obtained from it by removing a zero (viz. adding an edge) is stable. 

With this in mind, we prove the following facts about maximally unstable sparse matrix spaces: the minimal number of zero variables in an unstable SMS is $n$, that is every SMS with less than $n$ zero is stable. In fact, we will show that there is a unique maximally unstable sparse matrix space, we call it $\Sigma_{\alpha_0}$, with $n$ zeros: this is the space of matrices with zero diagonal entries. Next, we consider matrix spaces with one non-zero diagonal entry, which are obviously not included in $\Sigma_{\alpha_0}$. We show that these spaces have at least $n-1$ off-diagonal zeros; said otherwise, if a matrix space has at least one non-zero diagonal entry and at most $n-2$ off-diagonal zeros, then it is stable.

\subsection{Graphs and Hamiltonian decompositions}

We establish here the vocabulary used to deal with  graphs. We say that $G'=(V',E')$ is a subgraph of $G$, denoted by $G' \subset G$, if $G'$ is a graph and $V' \subset V$, $E' \subset E$. A class of subgraphs frequently used  consists of subgraphs induced by a vertex set: we call $G'=(V',E')$ the subgraph of $G$ induced by the vertex set $V'$ if $E'$ contains all edges of $G$ that both start and end at vertices of $V'$, that is $$E'=\lbrace (v_i,v_j) \mid v_i,v_j \in V' \mbox{ and } (v_i,v_j) \in E \rbrace.$$

We say that two subgraphs $G_1, G_2 \subset G$ are \emph{disjoint} if their vertex sets (and hence edge sets) are disjoint. We say that $G_1,\ldots, G_k\subset G$ is a \emph{decomposition of $G$} if the $G_i$'s are pairwise disjoint and  the union of the vertex sets of the $G_i$ is the vertex set of $G$. For the example of the graph depicted above, the subgraph induced by the vertex set $\{1,2,3\}$ contains the edges $\{(1,2),(2,3),(3,1)\}$. This subgraph is disjoint from the subgraph induced by the vertex set $\{4,5\}$, and they together form a decomposition of $G$.

Recall that a path of length $k$ in a graph is a sequence of vertices $v_1, \ldots,  v_k$ connected by edges, i.e. such that $(v_i,v_{i+1}) \in E$ for $i=1\ldots $k$-1.$ A \emph{cycle} is path with $v_1 = v_k$. A \emph{Hamiltonian cycle} is a cycle that visits every vertex exactly once, except for the origin vertex, which is visited twice~\cite{diestel_graph_2010}. Clearly, not all graphs admit Hamiltonian cycles.  We call a \emph{Hamiltonian decomposition} of the graph $G$ a decomposition of $G$ into disjoint subgraphs $G_1, \ldots, G_m$ where each $G_i$ admits a Hamiltonian cycle.  An important object in this work is the following:

\begin{Definition}Let $G=(V,E)$. A \emph{Hamiltonian subgraph} is a subgraph of $G$ which  admits a Hamiltonian decomposition. A Hamiltonian $k$-subgraph is a Hamiltonian subgraph of cardinality $k$.
\end{Definition}

We illustrate these definitions in Figure~\ref{fig:illhamdec}. 

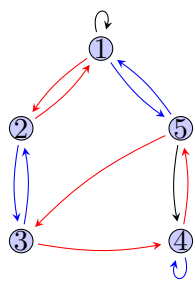
\begin{figure}
\begin{center}
\begin{tikzpicture}[scale = .25, ->,>=stealth,shorten >=1pt,auto,node distance=1.5cm,
  thin,main node/.style={circle,fill=blue!20,draw}]

  \node[main node] (1) {1};
  \node[main node] (2) [below left of=1] {2};
  \node[main node] (3) [below  of=2] {3};
  \node[main node] (5) [below right of=1] {5};
    \node[main node] (4) [below  of=5] {4};
  \path[every node/.style={font=\sffamily\small}]
    (1)      edge [bend right=10, red]  (2)
      edge [bend right=10, blue]  (5)
 edge [loop above] (1)   

    (2) edge [bend right=10,blue](3)
    (2) edge [bend right=10, red](1)
    
    (3)     edge [bend right=10,red](4)
 
    edge [bend right=10,blue](2)
(4) edge [bend right=10,red](5)
edge [loop below, right ,blue] (4)  

 (5)   edge [bend right=10,red](3)     
     edge [bend right=10,blue](1)     
     edge [bend right=10](4)  
        ;
\end{tikzpicture}
\end{center}
\caption{ The graph depicted above admits several Hamiltonian decompositions: one into the cycles $(12)$ and $(345)$, one into the cycles $(15), (23), (4)$ and one into the cycle $(12345)$. The subgraph induced by the vertices $\{1,2,3,4\}$ is a 4-Hamiltonian subgraph of $G$. }\label{fig:illhamdec}
\end{figure}

\subsection{Necessary and sufficient conditions}
We are now in a position to state further conditions to decide the stability of sparse matrix spaces:

\begin{Theorem}[Necessary condition for stability]
A sparse matrix space is stable only if  its associated graph contains, for each $k=1,\ldots,n$, a Hamiltonian $k$-subgraph. 
\label{th:mainthnec}
\end{Theorem}

\begin{Theorem}[Sufficient condition for stability]A sparse matrix space is stable if its associated graph $G$    contains a sequence of nested Hamiltonian subgraphs $G_1 \subset G_2 \subset  \ldots\subset G_{n-1}\subset G$.
\label{th:mainthsuff}
\end{Theorem}
Notice that the strict inclusions imply that the vertex set $V_k$ of $G_k$ has cardinality $k$.  From Theorem~\ref{th:mainthnec}, we conclude that  the graph on the left in Figure~\ref{fig:ma} is unstable, even though every node is strongly connected to a sink. From Theorem~\ref{th:mainthsuff}, we conclude that the graph on the right in Figure~\ref{fig:ma} is stable, a sequence of nested Hamiltonian subgraphs being given by the subgraphs induced by $\{1\}$ and $\{1,2\}$. We give further examples: consider the sparse matrix spaces

 { $$\Sigma_\alpha = \left[ \begin{array}{ccccc}
\ast & \ast & 0 & 0 & 0 \\
\ast & 0 & \ast & 0 & 0\\
\ast & 0 & 0& \ast & 0 \\
0 & 0 & \ast& 0& \ast \\
\ast & 0 & 0 & 0 & 0
\end{array}\right] 
\mbox{ and }
\Sigma_\beta = \left[ \begin{array}{ccccc}
\ast & \ast & 0 & 0 & \ast \\
0 & 0 & \ast & 0 & 0\\
\ast & 0 & 0& \ast & 0 \\
0 & 0 & 0& 0& \ast \\
\ast & 0 & 0 & \ast & 0
\end{array}\right]. $$} 

 The SMS $\Sigma_\alpha$ is stable since its associated graph, depicted in the Figure below left,  satisfies the conditions of Theorem~\ref{th:mainthsuff} with $G_k = \{ v_1,\ldots,v_k \}$. The Hamiltonian decompositions are respectively given by  $\left\lbrace \{v_1\}\right\rbrace $, $\left\lbrace \{v_1,v_2\}\right\rbrace $,  $\left\lbrace \{v_1,v_2, v_3 \}\right\rbrace$,$\left\lbrace \{v_1,v_2\}, \{v_3,v_4\} \right\rbrace$ and $\left\lbrace \{v_1,v_2,v_3,v_4,v_5 \}\right\rbrace $, each of these subgraphs having an obvious Hamiltonian cycle. On the contrary, the SMS $\Sigma_\beta$,  is not stable since the corresponding graph, depicted below right, contains  no Hamiltonian 4-subgraph. 

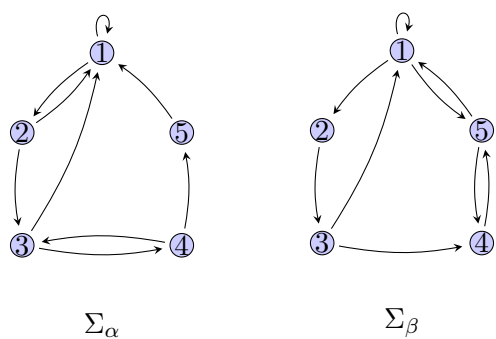
\begin{center}
\begin{tikzpicture}[scale = .25, ->,>=stealth,shorten >=1pt,auto,node distance=1.5cm,
  thin,main node/.style={circle,fill=blue!20,draw}]

  \node[main node] (1) {1};
  \node[main node] (2) [below left of=1] {2};
  \node[main node] (3) [below  of=2] {3};
  \node[main node] (5) [below right of=1] {5};
    \node[main node] (4) [below  of=5] {4};
    
    \node (name) [below right of=3]{$\Sigma_\alpha$};
  \path[every node/.style={font=\sffamily\small}]

    (1)      edge [bend right=10]  (2)
    
 edge [loop above] (1)   

    (2) edge [bend right=10](3)
   edge [bend right=10](1)
    
    (3) edge [bend right=10](1)
   		edge [bend right=10](4)
   		 (4) edge [bend right=10](5)
   		 edge [bend right=10](3)
   	(5)	edge [bend right=10](1)
        ;
        
\end{tikzpicture}\qquad\qquad
\begin{tikzpicture}[scale = .25, ->,>=stealth,shorten >=1pt,auto,node distance=1.5cm,
  thin,main node/.style={circle,fill=blue!20,draw}]

  \node[main node] (1) {1};
  \node[main node] (2) [below left of=1] {2};
  \node[main node] (3) [below  of=2] {3};
  \node[main node] (5) [below right of=1] {5};
    \node[main node] (4) [below  of=5] {4};
        \node (name) [below right of=3]{$\Sigma_\beta$};

  \path[every node/.style={font=\sffamily\small}]
    (1)      edge [bend right=10]  (2)
 edge [loop above] (1)   
edge [bend right=10]  (5)
    (2) edge [bend right=10](3)
    (5) edge [bend right=10](4)
    
    (3) edge [bend right=10](1)
   		edge [bend right=10](4)
   		 (4) edge [bend right=10](5)
   	(5)	edge [bend right=10](1)
        ;
\end{tikzpicture}
\end{center}

\section{The subspace of unstable SMS and its symmetries}\label{sec:groupactions}

We introduce in this section a subset of $\R^{n\times n}$  that contains all the unstable sparse matrix spaces. This subspace, which is  the common zero set of polynomial equations and is thus an algebraic set, is used extensively in the proof of the main results below. This subspace has a large number of symmetries, which we also study here.
We now exhibit three straightforward operations  that map stable SMS to stable SMS, a continuous operation and two discrete operations.

We start with the simpler discrete operation, which corresponds to matrix transposition. If $\alpha=\lbrace (a,b), (c,d), \ldots \rbrace$ is a sparse matrix space, we define 
\begin{equation}
\label{eq:deftranszp}
\alpha' = \lbrace (b,a), (d,c), \ldots \rbrace.
\end{equation}
In terms of matrix representation, if $A \in \Sigma_\alpha$, then $A' \in \Sigma_{\alpha'}$. In terms of graphical representation, it corresponds to changing the direction of all edges of the graph. That is,  $(v_i,v_j)$ is an edge of the graph associated to $\alpha$ if and only if $(v_j,v_i)$ is an edge of the graph associated to $\alpha'$. We record here the following obvious result:

\begin{Lemma}
The sparse matrix space $\alpha$ is stable  if and only if $\alpha'$ is stable.
\end{Lemma}

\subsection{Conjugation by permutation matrices}

Recall that the symmetric group $S_n$ acts on an ordered set of $n$ symbols by permuting its elements. The elements of $S_n$ are called permutations. Let $\sigma \in S_n$, we use the following standard notation to represent permutations:  we write $\sigma=(\sigma_1, \sigma_1, \ldots, \sigma_n)$, where $\sigma_i \in \lbrace 1,\ldots, n \rbrace$ and $\sigma_i \neq\sigma_j$ for all pairs $i \neq j$, to represent the permutation whose action on $\{1,2,\ldots,n\}$ moves the element in position $\sigma_i$ to position $i$. For example $\sigma=(1,3,4,2)$ acting on $(a,b,c,d)$ yields $(a,c,d,b)$.  We denote the composition of $\sigma, \tau \in S_n$ by $\sigma \tau$,  by $e$ the identity permutation and by $\sigma^{-1}$ the unique inverse permutation of $\sigma$.

We also use the representation of the symmetric group as $n$ by $n$ matrices: to each $\sigma \in S_n$, we associate the $n$ by $n$ matrix  $P_\sigma$ obtained  by permuting the columns of the identity matrix according to $\sigma$. For example
\begin{equation*}
\sigma=(1,3,4,2) \longrightarrow P_\sigma=\left[\begin{matrix} 1 & 0 & 0 &0\\ 0 & 0&0 & 1\\ 0 &1& 0 & 0\\0 & 0 & 1 & 0\end{matrix}\right].
\end{equation*}

We let the symmetric group act on a sparse matrix space $\alpha=\lbrace(a,b),(c,d), \ldots \rbrace$ via $$\sigma(\alpha)=\lbrace(\sigma_{a},\sigma_{b}),(\sigma_{c},\sigma_{d}), \ldots\rbrace
$$
In terms of matrix representation, this yields $$\Sigma_{\sigma(\alpha)} = P_\sigma \Sigma_{\alpha}P_\sigma^{-1}\triangleq \{ P_\sigma A P_\sigma^{-1}\mbox{ s.t. } A \in \Sigma_\alpha\}.$$ 
For example, for the $\Sigma_\alpha$ of~\eqref{ex:ma}, we have that $\Sigma_{\sigma(\alpha)}$ contains matrices of the form

\begin{equation}
\label{ex:msa}A = \left[\begin{matrix}
0 & 0&\ast & \ast \\
\ast & \ast & 0 &\ast \\
\ast & \ast & \ast & 0\\
0 &0 & \ast & 0
\end{matrix}\right]
\end{equation}
We say that $\sigma$ is \emph{even} (resp. \emph{odd}) if $\det(P_\sigma)=1$ (resp. $\det(P_\sigma)=-1$). We write $(-1)^\sigma$ for $\det(P_\sigma)$. Finally, in terms of the graphical representation, this action of $S_n$ corresponds to relabelling  the vertices and the edges following $$v_i \longrightarrow v_{\sigma_i} \mbox{ and }(v_i,v_j) \longrightarrow(v_{\sigma_i},v_{\sigma_j}).$$
We have the following result 

\begin{Lemma}\label{lem:stabsig}
A sparse matrix space $\Sigma_\alpha$ is stable (resp. unstable) if and only if $\Sigma_{\sigma(\alpha)}$ is stable (resp. unstable) for all $\sigma \in S_n$
\end{Lemma}

\begin{proof}
Assume $\Sigma_\alpha$ is stable and denote by $A$ a stable matrix in $\Sigma_\alpha$, then $P_\sigma A P_{\sigma}^{-1}$ is also stable and belongs to $\Sigma_{\sigma(\alpha)}$. Reciprocally, assume there exists a $\sigma$ such that $\Sigma_{\sigma(\alpha)}$ is stable and let $A$ be a stable matrix in $\Sigma_{\sigma(\alpha)}$, then $P_{\sigma}^{-1} A P_{\sigma}$ belongs to $\Sigma_\alpha$ and is stable.
\end{proof}

\subsection{Multiplication by non-singular diagonal matrices}

We observe that  multiplication to the left by a non-singular diagonal matrix maps a sparse matrix space to itself; that is \begin{equation}\label{eq:multdiag}D \Sigma_\alpha = \Sigma_\alpha,\end{equation} 
and similarly for multiplication to the right: $\Sigma_\alpha D = \Sigma_\alpha$.   

We say that a matrix $A$ is left- (resp. right- or both left-right-) \emph{diagonally stabilizable} if there exists a diagonal matrix $D$ such that $DA$ (resp. $AD$ or $D_1AD_2$)  is stable. We record the following fact about diagonally stabilizable matrices, which says that the sets of left-, right-, and left-right- diagonally stabilizable matrices are the same. 

\begin{Lemma}
The following statements are equivalent:
\begin{enumerate}
\item $A$ is left-diagonally stabilizable.
\item $A$ is right-diagonally stabilizable.
\item $A$ is left-right-diagonally stabilizable.
\end{enumerate}
\end{Lemma}

\begin{proof}
Let $A \in \R^{n \times n}$. First, observe that $1 \Rightarrow 3$ and $2 \Rightarrow 3$ are obvious. In view of this,  we  show that $3 \Rightarrow 2$ and $3 \Rightarrow 1$ and this  will prove the Lemma. To this end, assume that $D_1 A D_2$ is stable for invertible $D_1, D_2$. Then $D_2 D_1 A D_2 D_2^{-1}=D_2D_1 A$ is also stable and $D_2D_1$ is diagonal. Similarly, $AD_2D_1$ is stable.
\end{proof}
The above result justifies the terminology \emph{diagonally stabilizable}, without having to specify left or right. We have the following result, whose proof is obvious in view of~\eqref{eq:multdiag}:

\begin{Lemma}
A sparse matrix space is unstable if and only if it does not contain any diagonally stabilizable matrix.\label{lem:unst_diag}
\end{Lemma}

We now characterize the set of diagonally stabilizable matrices: for a subset  $\mathcal I \subset \lbrace 1,2,\ldots, n \rbrace$, we denote by $A_\mathcal{I}$ the principal submatrix of $A$ whose rows and columns are indexed by $\mathcal I$. We denote by $|\mathcal I|$ the cardinality of $\mathcal I$. A \emph{leading principal minor} is a \emph{principal minor} $\det(A_{\mathcal I})$ where $\mathcal I$ of the form $\{1,2,\ldots,k\}$, $k \in \{1,\ldots n\}$. We have the following result, taken from Yu, Anderson, Dasgupta and Fidan:
\begin{Lemma}\cite{yu_control_2009} If the leading principal minors of $A\in \R^{n \times n}$ are non-zero, then $A$ is diagonally stabilizable.
\label{th:diagstabander}
\end{Lemma}

The above condition is not necessary; a simple example of this fact is given by  the stable matrix 
\begin{equation}
\label{eq:counexdiagstabander}
A=\left[ \begin{matrix} 0 & -1\\2 & -1 \end{matrix} \right].
\end{equation}

For our purposes, we establish the following straightforward extension:

\begin{Corollary}
Let $A \in \R^{n\times n}$. If there exists $\sigma \in S_n$ such that the leading principal minors of $P_\sigma A P_{\sigma}^{-1}$ are non-zero, then $A$ is stable.
\label{cor:diagstab}
\end{Corollary}

\begin{proof}
Assume that there exists a permutation $\sigma \in S_n$ such that the leading principal minors of $P_\sigma A P_{\sigma}^{-1}$ are non-zero. We  exhibit a matrix $D$ such that $DA$ is stable. By Lemma~\ref{th:diagstabander}, there exists $D_1$ such that $D_1P_\sigma A P_{\sigma}^{-1}$ is stable. Observe that $D=P_\sigma^{-1} D_1 P_{\sigma}$ is a diagonal matrix and  $P_\sigma^{-1}(D_1P_\sigma A P_{\sigma}^{-1})P_\sigma=DA$ is stable (because similar to $D_1P_\sigma A P_{\sigma}^{-1}$).
\end{proof}
While the above result yields a stronger condition than Lemma~\ref{th:diagstabander}, and in particular recognizes the matrix in~\eqref{eq:counexdiagstabander} to be diagonally stabilizable, it is not a necessary and sufficient condition. 

\subsection{Unstable sparse matrix spaces}

Because they appear frequently in the remainder of the paper, we introduce a short-hand notation for leading principal minors: we denote by $\det_k(A)$ the $k$th leading principal minor of $A$. Hence $\det_1(A)=a_{11}, \det_2(A)=a_{11}a_{22}-a_{12}a_{21}$, etc. We also use the shorthand notation $\det_i(\sigma(A))=\det_i(P_\sigma A P_\sigma^{-1})$.

Let $\sigma \in S_n$, we introduce the polynomial $p_\sigma(A)$:
\begin{equation}\label{eq:defps}
p_\sigma(A)={\sideset{}{_1}\det} (\sigma(A)) {\sideset{}{_2}\det}(\sigma(A))\ldots {\sideset{}{_{n-1}} \det}(\sigma(A)).
\end{equation}
That is, $p_\sigma(A)$ is the product of the leading principal minors of $P_\sigma A P_{\sigma}^{-1}$.

We introduce the algebraic set \begin{equation}\label{eq:defV}{\cal V} = \left\lbrace A \in \R^{n \times n} \mbox{ s.t. } p_\sigma(A)=0 \mbox{ for all } \sigma \in S_n \right\rbrace.\end{equation}
As a consequence of Lemma~\ref{lem:unst_diag}, we have the following result:

\begin{Theorem}
If  $\Sigma_\alpha$ is an unstable sparse matrix space, then $\Sigma_\alpha \subset {\cal V}$.\label{th:unstabinU}
\end{Theorem}

We summarize in the following proposition the symmetries of ${\cal V}$:

\begin{Proposition}
Let $A \in {\cal V}$. Then
\begin{enumerate}
\item $A'\in {\cal V}$
\item $DA \in {\cal V}$ for all invertible  diagonal matrices $D$
\item $P_\tau A P_\tau^{-1} \in {\cal V}$ for any $\tau \in S_n$
\item if $A \in {\cal V}$ is invertible, $A^{-1} \in {\cal V}$
\end{enumerate}

\end{Proposition}

\begin{proof} If $A \in {\cal V}$, then for all $\sigma \in S_n$, $p_{\sigma}(A)=0$. For the first statement, we have $\det_i(A)=\det_i(A')$, which implies that  $p_\sigma(A)=p_\sigma(A')$ for any permutation $\sigma$, which proves the claim. For the  second statement,  observe that  $\det((DA)_{\mathcal I})=\det(D_{\mathcal I})\det(A_{\mathcal I})$ for all subset $\mathcal I$ of $\{1,2,\ldots,n\}$. Hence $p_\sigma(DA) = k p_{\sigma}(A)$ where $k$ is the product of the leading principal minors of $P_\sigma DP_{\sigma}^{-1}$ and we conclude that $p_\sigma(A)=0 \Rightarrow p_\sigma(DA)=0.$
For the third statement, observe that, from the definition of the polynomials $p_\tau(A)$, we have that   $p_\tau(P_\sigma AP_\sigma^{-1})=p_{\tau\sigma}(A)$ which is zero because $A$ is in ${\cal V}$.

We now focus on the last statement of the Proposition.  We show that for every $\sigma \in S_n$, there exists an integer $k_\sigma$ such that $\det_{k_\sigma}(\sigma(A^{-1}))=0$. This  implies that $p_\sigma(A^{-1})=0$ as required.

To this end, we recall the Jacobi identity: let $B$ be an invertible matrix. The Jacobi identity relates the principal minors of $B$ and the principal minors of $B^{-1}$. Precisely, it says that $${\det}(B_{\mathcal{I}}^{-1}) = \frac{1}{\det(B)}{\det}B_{\mathcal{I}^c}.$$ Now, introduce the shorthand notation $$ \sideset{}{_{ \bar j}}{\det}(A)=\det(A_{\{j+1,j+2,\ldots,n\}}),$$
and from the Jacobi identity we deduce that if $\det_j(A)=0$, then $\det_{\bar j}(A^{-1})=0$.  Now because $p_\sigma(A)=0$ for all $\sigma \in S_n$,  there exist  integers $k_\sigma$ such that $\det_{k_\sigma}(\sigma(A))=0$. Observe that $P_\sigma A^{-1}P_\sigma^{-1}=(P_\sigma  AP_\sigma^{-1})^{-1}$ and hence the Jacobi identity yields
\begin{equation}
\label{eq:intexprop1}
\sideset{}{_{\bar{k}_\sigma}}{\det}(\sigma(A^{-1}))=0
\end{equation}
Denote by $\eta$ the permutation that reverses the order of $n$ elements: $\eta=(n,n-1,\ldots,1)$.  Now given a permutation $\sigma$ written in one-line notation, we define $\tau$ as \begin{equation}\label{eq:deftauprop}\tau =\eta\sigma= \left(\sigma_{n}, \sigma_{n-1}, \ldots, \sigma_{2},\sigma_1\right). \end{equation}
The first $n-k$ elements of $\tau$ are the last $n-k$ elements of $\sigma$ in reverse order. Since the absolute value of a determinant is invariant under row and column permutations, we conclude that
 we conclude that \begin{equation}\label{eq:propconc}\sideset{}{_{\bar{k}_\sigma}}{\det}(\sigma(A^{-1}))=0 \Leftrightarrow \sideset{}{_{n-k_\sigma}}{\det}(\tau(A^{-1}))=0\end{equation}
Because~\eqref{eq:intexprop1} holds for every $\sigma \in S_n$, we conclude that~\eqref{eq:propconc} holds for every $\tau \in S_n$ and thus $A^{-1} \in {\cal V}$.

\end{proof}

From the above results, we conclude that the space $${\cal U} \triangleq {\cal S} \cap {\cal V}$$ contains all the unstable sparse matrix spaces.

\section{Stable and unstable graphs}\label{sec:hamildec}

We now prove Theorems~\ref{th:strongsink},~\ref{th:mainthnec} and~\ref{th:mainthsuff}.  Recall that a sink is a node with a self-loop.

\begin{proof}[Proof of Theorem~\ref{th:strongsink}]
Let $G=(V,E)$ be the graph associated to the sparse matrix space   and assume without loss of generality that it is connected. Let $v \in V$ and let $V_c$ be the strongly connected component to which $v$ belongs. If $V_c$ is empty, and in particular $v$ is not a sink, take $V_c=\{v\}$. We show that if the subgraph induced by $V_c$ has no sink, then $G$ is not stable.

To this end,  split $V - V_c$ into $V_{\text{in}}$, consisting of vertices $v_j$ such such that there exists a path from $v_j$ to at least one vertex $u_i \in V_c$,  and  $V_{\text{out}}$ consisting of vertices  $w_i$ such that there exist a path from at least one $u_i \in V_c$ to $w_i$.

\begin{figure}
\begin{center}
\begin{tikzpicture}[scale = .25, ->,>=stealth,shorten >=1pt,auto,node distance=1.5cm,
  thin,main node/.style={circle,fill=blue!20,draw}]

  \node[main node] (1) {1};
  \node[main node] (2) [below right of=1,yshift=.33cm] {2};
  \node[main node] (3) [below left of=2,xshift=.5cm] {3};
  \node[main node] (5) [above right of=2,yshift=-.33cm] {5};
    \node[main node] (4) [right  of=3,xshift=-.5cm] {4};
  \path[every node/.style={font=\sffamily\small}]
    (1)      edge [bend left=10]  (2)
 edge [bend right=10] (3)   

    (2) edge [bend right=10](3)
    (2) edge [bend left=10](5)
        
    (4) edge [bend right=10](2)
    
    (3) edge [bend right=10](4)
   (4) edge [bend right=10](5)
   (5) edge [loop above] (5)
        ;
        \draw[color=black!90,fill=black!10,opacity=.4]  (4.3,-1) .. controls +(left:1.5cm) and +(left:1.5cm) .. (0,-8) .. controls +(left:1.5cm) and +(left:1.5cm) .. (8.3,-8.01) .. controls +(right:1.5cm) and +(right:1.5cm) .. (4.3,-1) -- cycle;
\end{tikzpicture}\quad
 \begin{tikzpicture}
 \node at (-2,0) {$\Longleftrightarrow$};
        \matrix [matrix of math nodes,left delimiter=(,right delimiter=),row sep=2mm,column sep=2mm] (m)
        {
            0 & \ast & \ast & 0 & 0\\
           0 & 0 &\ast & 0 &\ast\\
           0 &0 &0 &\ast& 0\\
           0 & \ast & 0 & 0 & \ast\\
           0 &0 &0 &0 &\ast\\
        };  
        \draw[color=black!90,fill=black!10,opacity=.4] (m-2-2.north west) -- (m-2-4.north east) -- (m-4-4.south east) -- (m-4-2.south west) -- (m-2-2.north west);
        
    \end{tikzpicture}

\end{center}
\caption{ The strongly connected component $\{2,3,4\}$ has no sink and thus the graph is not stable according to Theorem~\ref{th:strongsink}. With the notation of the proof of Theorem~\ref{th:strongsink}, $V_{\text{in}} = \{1\}$ and $V_{\text{out}}=\{5\}$. The corresponding sparse matrix space, illustrated on the right, is upper block diagonal since there are no edges linking $V_{\text{out}}$ to $V_{\text{in}}$. The block corresponding to the component $\{2,3,4\}$ has a zero diagonal and thus cannot have strictly negative eigenvalues. }\label{fig:illTh1}
\end{figure}
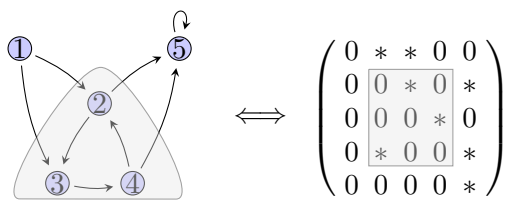
There are no edges from $V_{\text{out}}$ to $V_{\text{in}}$, since this would contradict the maximality of $V_c$. Order the vertices such that $V_{\text{in}}=\lbrace 1,2,\ldots,k\rbrace$, $V_c = \lbrace k+1,\ldots, l \rbrace$ and $V_{\text{out}} = \lbrace l+1,\ldots n\rbrace$. From Lemma~\ref{lem:stabsig}, we know that the stability properties of the graph with reordered vertices are equivalent to the stability properties of the original graph. Because there are no edges from $V_{\text{out}}$ to $V_{\text{in}}$,  the sparse matrix space corresponding to the graph with vertices reordered as above consists of upper block triangular matrices. (see Figure~\ref{fig:illTh1} for an illustration). The  eigenvalues of such matrices are the eigenvalues of the diagonal blocks. Since the diagonal block corresponding to $V_c$ has a zero diagonal--because no vertex is a sink -- it cannot be stable, and we conclude that no matrix in that space is stable.
\end{proof}

\subsection{Hamiltonian decompositions and stability}
The remaining  results rely on a correspondence between Hamiltonian decompositions of graphs and permutations of $S_n$ that we explain here.  Consider the permutation group $S_n$ acting on the set $\mathcal N=\lbrace 1,2,\ldots, n \rbrace$. A \emph{cycle} is a permutation that maps the elements of some subset $\mathcal N_1 \subset \mathcal N$ to each other in a cyclic fashion, while leaving the other elements fixed. For example, the permutation $(3,1,2,4)$ is a cycle since it leaves $4$ fixed, and maps the elements of $S=\lbrace 1,2,3\rbrace$ to each other in a cyclic fashion, but the permutation $(2,1,4,3)$ is not a cycle.

We adopt the widely used convention of denoting a cycle by $\bold{i}=(i_1i_2,\ldots i_k)$, where the $i_k$ are pairwise different,  to indicate that the element in position $i_1$ is replaced by the element in position $i_2$, the element in $i_2$ by the one in $i_3$ all the way to $i_n$ by $i_1$ while the other elements are fixed. We say that two cycles $\bold{i}$ and $\bold{j}$ are \emph{disjoint} if $i_l \neq j_m$ for all $l,m$. We call $k$ the \emph{order} of a cycle and we refer to cycles of order $k$ as $k$-cycles. It is a well-known fact that any permutation can be written as the composition of disjoint cycles. For example, the permutation that sends $\lbrace 1,2,3,4 \rbrace$ to $\lbrace 2,1,4,3 \rbrace$ is the composition of $(12)$ and $(34)$ and is written as $(12)(34)$.  It is easy to see that disjoint cycles  commute.

Now, observe that $k$-cycles in $S_n$ can be put in one-to-one correspondence with Hamiltonian cycles going  through $k$ vertices  in a graph on $n$ vertices: the cycle  $(i_1i_2i_3...)$ is mapped to the Hamiltonian cycle that start at vertex $i_1$, then goes to vertex $i_2$, etc. and vice-versa. In the same vein, to a Hamiltonian decomposition we can assign the composition of the necessarily disjoint cycles corresponding to each Hamiltonian cycle. Because the different Hamiltonian cycles in a Hamiltonian decomposition do not share vertices, the resulting cycles commute and the assignment of a Hamiltonian decomposition to a composition of cycles in $S_n$ is well-defined. Returning to the example of Figure~\ref{fig:illhamdec}, the Hamiltonian decomposition into $\{1,2\}$ and $\{3,4,5\}$ of the graph corresponds to the cycles $(12)$ and $(345)$ whose composition is the permutation $\sigma=(2,1,4,5,3)$.

\begin{proof}[proof of Theorem~\ref{th:mainthnec}]
We show that all the matrices in a sparse matrix space $\Sigma_\alpha$ that satisfies the conditions of the theorem have a characteristic polynomial with a zero coefficient, and hence are not stable (see~\cite{kaplan_operational_book_62}, Ch. 7 Th. 1).  To this end, write $p(A)=s^n+p_1s^{n-1}+p_2^{n-2} s^{n-2}+\ldots+p_n$ for the characteristic polynomial of $A$ and denote by $G$ the graph associated to $\Sigma_\alpha$. It is well-known that
\begin{multline}
p_1 = - \sum_{i=1}^n a_{ii}; 
p_2 = \sum_{|\mathcal I|=2} \det(A_\mathcal{I}); \cdots ;\\
 p_k =  (-1)^k \sum_{|\mathcal{I}|=k} \det(A_{\mathcal{I}}); \cdots;
p_n = (-1)^n \det(A)
\end{multline}

where the sums are taken over all subsets $\mathcal I \subset \mathcal N$ of given cardinality. 

In order for $p_k$ to be different from, there has to be at least one subset $\mathcal I$ of cardinality $k$ such that $\det(A_\mathcal{I})$ is non-zero. Let  $\mathcal{I}^*$ be one such subset. We can  express    $\det(A_{\mathcal{I}^*})$  as

$$\det A_{\mathcal{I}^*} = \sum_{\sigma \in S_k} \prod_{i\in \mathcal{I}^*} a_{i\sigma(i)}. $$ 
Hence,  there exists at least one permutation, say  $\sigma^* \in S_k$,  such that $\prod_{i\in \mathcal{I}^*} a_{i\sigma^*(i)}$ is a product of free variables. Because free variables correspond to edges in $G$, the subgraph of $G$ on the vertex set $\lbrace v_i, i \in \mathcal I^*\rbrace$ has a Hamiltonian decomposition--the decomposition corresponding to the permutation $\sigma^*$ per the above discussion. Thus if there are no subgraphs of size $k$ admitting a Hamiltonian decomposition, then $p_k=0$. This proves the theorem.
\end{proof}

Before moving on to the proof of Theorem~\ref{th:mainthsuff}, we  relate the existence of non-singular matrices in a sparse matrix space to properties of the associated graph:
\begin{Lemma}\label{lem:hamiltdeczerodet}
The sparse matrix space associated to a graph admitting a Hamiltonian decomposition contains  non-singular matrices.
\end{Lemma}
\begin{proof}
Let $A=(a_{ij})$ be a matrix in the sparse matrix space. We can express  its determinant as $$\det(A)=\sum_{\sigma\in S_n} (-1)^{\sigma}\prod_{i=1}^n a_{i\sigma(i)}.$$ Because the associated graph admits a Hamiltonian decomposition, we know  from the discussion at the beginning of Section~\ref{sec:hamildec} that there is at least one permutation $\sigma$ for which $\prod_{i=1}^n a_{i\sigma(i)}$ is the product of free variables. Pick one such permutation,  denote it $\sigma^*$ and assume without loss of generality that it is even. We assign values to the free variables that make the matrix non-singular. To this end, let $a_{i\sigma^*(i)}=n!, i=1\ldots n$ and let all the other free variables be equal to $1$. For two different permutations $\sigma$ and $\tau$, the products $\prod_{i=1}^n a_{i\sigma(i)}$ and $\prod_{i=1}^n a_{i\tau(i)}$ contain at least a pair of terms $a_{i\sigma(i)}$ and $a_{i\tau(i)}$ which are not the same, otherwise $\tau=\sigma$. This yields the inequalities $\prod_{i=1}^n a_{i\sigma(i)} \leq (n!)^{n-1}$ for $\sigma \neq \sigma^*$, which allow us to lower bound the determinant of $A$ as follows
\begin{equation*}
\det(A) = \prod_{i=1}^n a_{i\sigma^*(i)}  + \sum_{\sigma\neq \sigma^*} (-1)^{\sigma}\prod_{i=1}^n a_{i\sigma(i)} > (n!)^n - (n!-1) (n!)^{n-1} > 0
\end{equation*} where we recall that $S_n$ contains $n!$ elements. This concludes the proof.
\end{proof}
We now give the proof of Theorem~\ref{th:mainthsuff}:

\begin{proof}[proof of Theorem~\ref{th:mainthsuff}]
The approach we take to prove the result is the following: we show that a sparse matrix space with associated graph satisfying the hypothesis of the theorem contains matrices that are diagonally stabilizable, and hence stable,  by showing that it is \emph{not included} in  $\cal V$ of~\eqref{eq:defV}.

Let $i_1, i_2, \ldots, i_n$ be integers such that the vertex set of $G_1$ is $\lbrace v_{i_1}\rbrace$, the vertex set of $G_2$ is $\lbrace v_{i_1}, v_{i_2}\}$, etc. and set $\sigma=(i_1, i_2, \ldots, i_n)$. We claim that we can  find a matrix $A$ in the sparse matrix space such that $p_\sigma(A) \neq 0$, which will prove the Theorem.

Since $G_1$ has a Hamiltonian decomposition with one vertex, it means that $a_{i_1i_1}$ is a free variable. Thus $\det_1(\sigma(A))$ vanishes on a set of codimension $1$ in $\Sigma_\alpha$.  More generally, since $G_i$ admits a Hamiltonian decomposition, from Lemma~\ref{lem:hamiltdeczerodet} we know that there are matrices in $\Sigma_\alpha$ with non-zero determinant, that is $\det_i(\sigma(A))$ does not vanish identically on $\Sigma_\alpha$.  Since the determinant is a polynomial function, we conclude that $\det_i(\sigma(A))$ vanishes on a set of codimension at least one in $\Sigma_\alpha$. 

Putting this together, for $p_\sigma(A)$ to vanish, at least one of its factors needs to vanish, and we have just shown that this happens on the union of sets of codimension one in $\Sigma_\alpha$. Hence,   there are matrices in $\Sigma_\alpha$ for which $p_\sigma(A)\neq 0$ and $\Sigma_a \nsubseteq {\cal V}$
\end{proof}

\section{Maximal and minimal sparse matrix spaces}\label{sec:setv}
We now prove some results about the structure of the set $\cal S$ of sparse matrix spaces. In particular, we show that the minimal codimension of an unstable sparse matrix space with one free variable on the diagonal is $2n-2$. That is, if an unstable sparse matrix space contains one non-zero diagonal, then it contains  at least $n-1$ off-diagonal zeros. Further results concerning this phenomenon will appear in a forthcoming publication.  As already mentioned earlier in this paper, the following relations  hold:
$$\Sigma_\alpha \mbox{ is stable } \Longrightarrow \mbox{ all } \Sigma_\beta \mbox{ such that } \Sigma_\alpha \subset \Sigma_\beta \mbox{ are stable}$$ 
and reciprocally
$$\Sigma_\alpha \mbox{ is unstable } \Longrightarrow \mbox{ all } \Sigma_\beta \mbox{ such that } \Sigma_\beta \subset \Sigma_\alpha \mbox{ are unstable}.$$
These inclusions are also easily visualized in terms of the graphs associated to sparse matrix spaces: $\Sigma_\alpha \subset \Sigma_\beta$ means that $\Sigma_\beta$ contains all the edges of $\Sigma_\alpha$ and some additional edges. Recall that the dimension of a sparse matrix space is the dimension of the vector space $\Sigma_\alpha$ and its codimension is $n^2-\dim \Sigma_\alpha$.  In view of the above facts, we introduce the following definition:
\begin{Definition} \begin{enumerate}
\item We call a sparse matrix space $\Sigma_\alpha$ \emph{minimally stable} if it is stable and there are no  $\Sigma_\beta \subset \Sigma_\alpha$ that are stable.
\item We call a sparse matrix space $\Sigma_\alpha$ \emph{maximally unstable} if it is unstable and there are no  $\Sigma_\beta \supset \Sigma_\alpha$ that are unstable.
\end{enumerate}
\end{Definition}
Equivalently, a minimally stable graph is a stable graph which can spare no edge while remaining stable, and a maximally unstable graph is an unstable graph which becomes stable if any edge is added to its edge set.
We have the following results:

\begin{Theorem} \label{th:structV}
\begin{enumerate}
\item The least dimension of a minimally stable sparse matrix space is $n$.
\item The least codimension of a maximally unstable sparse matrix space is $n$.
\item There are no maximally unstable sparse matrix space with one non-zero diagonal and codimension between $n$ and $2n-2$. .
\end{enumerate}

\end{Theorem}
\begin{proof}
{\bf 1.} Observe that diagonal matrices form a stable SMS of dimension $n$ and since every SMS of dimension strictly less than $n$ contains only rank deficient matrices, it is not stable. We denote by $\Sigma_\delta$ the SMS of diagonal matrices.

{\bf 2.} It is clear that the SMS $\Sigma_0 $ with zero variables on the diagonal   is unstable. We  show here that any SMS with less than $n$ zero variables is stable, which proves the maximality of $\Sigma_0$. To this end, let $G=(V,E)$ be the graph associated to an arbitrary SMS $\Sigma$ with $n-1$ zero variables. We show by induction that the conditions of Theorem~\ref{th:mainthsuff} are satisfied and thus $\Sigma$ is stable. We start by observing that  $G$ has at least one sink, say $v_1$. We claim that there exists at least one vertex, say $v_2$, such the subgraph  induced by $\{v_1,v_2\}$ (call it $G_2$) admits a Hamiltonian decomposition. We consider two cases: first,  if at least one  vertex besides $v_1$ is a sink, then take $v_2$ to be that vertex and we are done. Second, assume that no vertex other than $v_1$ is a sink. Then all the $n-1$ zero variables of $\Sigma$ are on the diagonal and, in particular, all the edges $(v_i,v_j)$ for $i\neq j$ are in $E$. Thus $G_2$ admits a Hamiltonian cycle. This proves the base case. 

For the inductive step,   assume that the subgraph induced by $\{v_1,\ldots, v_k\}$ (call it $G_k$) admits a Hamiltonian decomposition. We will show that there exists a vertex $v$ such that the subgraph  induced by $\{v_1, \ldots, v_k,v\}$ admits a Hamiltonian decomposition. As above, we consider two cases: first, if any vertex in $\{v_{k+1}, \ldots, v_n\}$ is a sink, then we can take $v$ to be that vertex and we are done. Second, assume that none of these vertices are sinks, then $\Sigma$ has  at least $n-k$ zero variables on the diagonal.
 Let $w=(w_1w_2\ldots w_l)$ be a Hamiltonian cycle in the Hamiltonian decomposition of $G_k$.  If there is a pair of vertices $w_i$, $v$ for which both edges $(w_i,v)$ and $(v,w_{i+1})$ are in $G$, then we can extend the cycle $w$ to a cycle of length $l+1$ (see Figure~\ref{fig:exproofthegap}). Because   every vertex in $G_k$ belongs to a cycle, there are $k(n-k)$ candidate pairs. By a  simple counting argument, using the fact that no edge appears more than once in $\{(w_i,v), (v,w_{i+1})\}$, there needs to be at least $k(n-k)+n-k$ zero variables to prevent the existence of $v$ such that $\{v_1,\ldots,v_k,v\}$ has a Hamiltonian decomposition. Because $k(n-k)>n-1$ for any $k \in \{1,2,\ldots,n-1\} $,   we see that this value is larger than $n$.

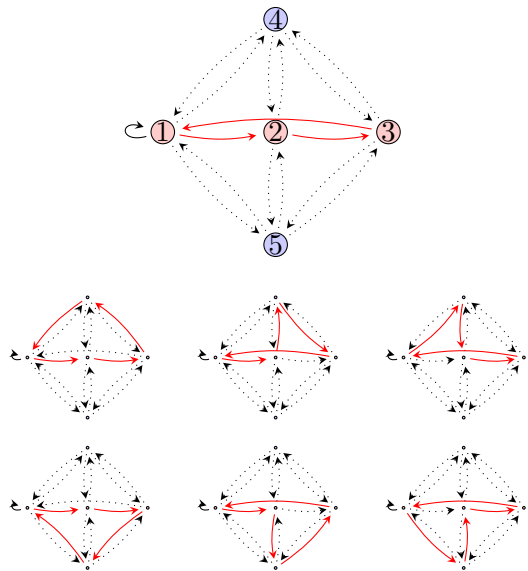
\begin{figure}
\begin{center}

\begin{tikzpicture}[scale = .1, ->,>=stealth,shorten >=1pt,auto,node distance=1.5cm,
  thin,main node/.style={circle,fill=blue!20,draw},main nr/.style={circle,fill=red!20,draw}]
\begin{scope}[xshift=-7cm,yshift = 0cm]
  \node[main nr] (1) {1};
  \node[main nr] (2) [right  of=1] {2};
  \node[main nr] (3) [ right of=2] {3};
  \node[main node] (4) [above of=2] {4};
  \node[main node] (5) [below of=2] {5};

  \path[every node/.style={font=\sffamily\small}]
    (1)      edge [bend right=10,red]  (2)
   edge   [bend right=10,  dotted] (4)
   edge   [bend right=10, dotted] (5)
 edge [loop left]  (1)

    (2) edge [bend right=10,red](3)
    edge [bend right=10, dotted](4)
    edge [bend right=10, dotted](5)
       
    (3)edge [bend right=10,red](1)
    edge [bend right=10, dotted](4)
    edge [bend right=10, dotted](5)

    (4)edge [bend right=10, dotted](3)
    edge [bend right=10, dotted](2)
    edge [bend right=10, dotted](1)
        (5)edge [bend right=10, dotted](3)
    edge [bend right=10, dotted](2)
    edge [bend right=10, dotted](1);
  \end{scope}
  \begin{scope}[ yshift=-30cm, xshift=-25cm,->,>=stealth,shorten >=1pt,auto,node distance=.8cm,
  thin,main node/.style={circle,fill=blue!20,draw},main nr/.style={circle,fill=red!20,draw}]
    \node[main nr] (1) {};
  \node[main nr] (2) [right  of=1] {};
  \node[main nr] (3) [ right of=2] {};
  \node[main node] (4) [above of=2] {};
  \node[main node] (5) [below of=2] {};

  \path[every node/.style={font=\sffamily\small}]
    (1)      edge [bend right=10,red]  (2)
   edge   [bend right=10, dotted] (4)
   edge   [bend right=10, dotted] (5)
 edge [loop left]  (1)

    (2) edge [bend right=10,red](3)
    edge [bend right=10, dotted](4)
    edge [bend right=10, dotted](5)
       
    (3)edge [bend right=10,dotted](1)
    edge [bend right=10, red](4)
    edge [bend right=10, dotted](5)

    (4)edge [bend right=10, dotted](3)
    edge [bend right=10, dotted](2)
    edge [bend right=10, red](1)
        (5)edge [bend right=10, dotted](3)
    edge [bend right=10, dotted](2)
    edge [bend right=10, dotted](1);
    \end{scope}
    
  \begin{scope}[ yshift=-30cm, xshift=0cm,->,>=stealth,shorten >=1pt,auto,node distance=.8cm,
  thin,main node/.style={circle,fill=blue!20,draw},main nr/.style={circle,fill=red!20,draw}]
    \node[main nr] (1) {};
  \node[main nr] (2) [right  of=1] {};
  \node[main nr] (3) [ right of=2] {};
  \node[main node] (4) [above of=2] {};
  \node[main node] (5) [below of=2] {};

  \path[every node/.style={font=\sffamily\small}]
    (1)      edge [bend right=10,red]  (2)
   edge   [bend right=10, dotted] (4)
   edge   [bend right=10, dotted] (5)
 edge [loop left]  (1)

    (2) edge [bend right=10,dotted](3)
    edge [bend right=10, red](4)
    edge [bend right=10, dotted](5)
       
    (3)edge [bend right=10,red](1)
    edge [bend right=10, dotted](4)
    edge [bend right=10, dotted](5)

    (4)edge [bend right=10, red](3)
    edge [bend right=10, dotted](2)
    edge [bend right=10, dotted](1)
        (5)edge [bend right=10, dotted](3)
    edge [bend right=10, dotted](2)
    edge [bend right=10, dotted](1);
    \end{scope}
        \begin{scope}[ yshift=-30cm, xshift=25cm,->,>=stealth,shorten >=1pt,auto,node distance=.8cm,
  thin,main node/.style={circle,fill=blue!20,draw},main nr/.style={circle,fill=red!20,draw}]
    \node[main nr] (1) {};
  \node[main nr] (2) [right  of=1] {};
  \node[main nr] (3) [ right of=2] {};
  \node[main node] (4) [above of=2] {};
  \node[main node] (5) [below of=2] {};

  \path[every node/.style={font=\sffamily\small}]
    (1)      edge [bend right=10,dotted]  (2)
   edge   [bend right=10, red] (4)
   edge   [bend right=10, dotted] (5)
 edge [loop left]  (1)

    (2) edge [bend right=10,red](3)
    edge [bend right=10, dotted](4)
    edge [bend right=10, dotted](5)
       
    (3)edge [bend right=10,red](1)
    edge [bend right=10, dotted](4)
    edge [bend right=10, dotted](5)

    (4)edge [bend right=10, dotted](3)
    edge [bend right=10, red](2)
    edge [bend right=10, dotted](1)
        (5)edge [bend right=10, dotted](3)
    edge [bend right=10, dotted](2)
    edge [bend right=10, dotted](1);
    \end{scope}
       
         
  \begin{scope}[ yshift=-50cm, xshift=-25cm,->,>=stealth,shorten >=1pt,auto,node distance=.8cm,
  thin,main node/.style={circle,fill=blue!20,draw},main nr/.style={circle,fill=red!20,draw}]
    \node[main nr] (1) {};
  \node[main nr] (2) [right  of=1] {};
  \node[main nr] (3) [ right of=2] {};
  \node[main node] (5) [above of=2] {};
  \node[main node] (4) [below of=2] {};

  \path[every node/.style={font=\sffamily\small}]
    (1)      edge [bend right=10,red]  (2)
   edge   [bend right=10, dotted] (4)
   edge   [bend right=10, dotted] (5)
 edge [loop left]  (1)

    (2) edge [bend right=10,red](3)
    edge [bend right=10, dotted](4)
    edge [bend right=10, dotted](5)
       
    (3)edge [bend right=10,dotted](1)
    edge [bend right=10, red](4)
    edge [bend right=10, dotted](5)

    (4)edge [bend right=10, dotted](3)
    edge [bend right=10, dotted](2)
    edge [bend right=10, red](1)
        (5)edge [bend right=10, dotted](3)
    edge [bend right=10, dotted](2)
    edge [bend right=10, dotted](1);
    \end{scope}
    
  \begin{scope}[ yshift=-50cm, xshift=0cm,->,>=stealth,shorten >=1pt,auto,node distance=.8cm,
  thin,main node/.style={circle,fill=blue!20,draw},main nr/.style={circle,fill=red!20,draw}]
    \node[main nr] (1) {};
  \node[main nr] (2) [right  of=1] {};
  \node[main nr] (3) [ right of=2] {};
  \node[main node] (5) [above of=2] {};
  \node[main node] (4) [below of=2] {};

  \path[every node/.style={font=\sffamily\small}]
    (1)      edge [bend right=10,red]  (2)
   edge   [bend right=10, dotted] (4)
   edge   [bend right=10, dotted] (5)
 edge [loop left]  (1)

    (2) edge [bend right=10,dotted](3)
    edge [bend right=10, red](4)
    edge [bend right=10, dotted](5)
       
    (3)edge [bend right=10,red](1)
    edge [bend right=10, dotted](4)
    edge [bend right=10, dotted](5)

    (4)edge [bend right=10, red](3)
    edge [bend right=10, dotted](2)
    edge [bend right=10, dotted](1)
        (5)edge [bend right=10, dotted](3)
    edge [bend right=10, dotted](2)
    edge [bend right=10, dotted](1);
    \end{scope}
        \begin{scope}[ yshift=-50cm, xshift=25cm,->,>=stealth,shorten >=1pt,auto,node distance=.8cm,
  thin,main node/.style={circle,fill=blue!20,draw},main nr/.style={circle,fill=red!20,draw}]
    \node[main nr] (1) {};
  \node[main nr] (2) [right  of=1] {};
  \node[main nr] (3) [ right of=2] {};
  \node[main node] (5) [above of=2] {};
  \node[main node] (4) [below of=2] {};

  \path[every node/.style={font=\sffamily\small}]
    (1)      edge [bend right=10,dotted]  (2)
   edge   [bend right=10, red] (4)
   edge   [bend right=10, dotted] (5)
 edge [loop left]  (1)

    (2) edge [bend right=10,red](3)
    edge [bend right=10, dotted](4)
    edge [bend right=10, dotted](5)
       
    (3)edge [bend right=10,red](1)
    edge [bend right=10, dotted](4)
    edge [bend right=10, dotted](5)

    (4)edge [bend right=10, dotted](3)
    edge [bend right=10, red](2)
    edge [bend right=10, dotted](1)
        (5)edge [bend right=10, dotted](3)
    edge [bend right=10, dotted](2)
    edge [bend right=10, dotted](1);
    \end{scope}
\end{tikzpicture}
\end{center}
\caption{{ The subgraph $G_3$ induced by the vertices $1,2,3$ has the Hamiltonian cycle $(123)$ and is fully connected to the remaining vertices $4$ and $5$. The subgraph $G_4^4$, induced by the vertex set $\{1,2,3,4\}$, has three Hamiltonian cycles as illustrated in row right below the larger graph; the cycles are $(1234), (1243), (1423)$. None of the newly added edges in  $G_4^4$ (newly added compared to the edge set of $G_3$) appear more than once in the cycles. Hence in order to break these three cycles, we need to remove 3 of the newly added edges at least. The bottom row  deals with the subgraph induced by the vertex set $\{1,2,3,5\}$, that is $G_4^5$. The conclusion is the same as for $G_4^4$ and we thus need to subtract at least 6 edges to prevent having  a Hamiltonian cycle of length 4 by adding vertex to $G_3$. }}
\label{fig:exproofthegap}
\end{figure}

{\bf 3.} We again work with both the matrix and graphical representation of a SMS:  consider the graph $G=(V,E)$ on $n$ vertices and with the following edge set $$E=\lbrace (v_1,v_1)\rbrace \cup \lbrace (v_i,v_j), i \neq j \rbrace,$$ that is $G$ is a complete graph with a self-loop on $v_1$, and no other self-loops. This graph is easily seen to correspond to a SMS with $n-1$ zero diagonals and all other entries free. We show that any graph $G'$ obtained from $G$ by subtracting at most $n-2$ edges satisfies the conditions of Theorem~\ref{th:mainthsuff} and is thus stable.

To this end, we show that, starting from $G_1=\{v_1\}$, we can construct a series of subgraphs $G_1 \subset G_2 \subset \ldots \subset G_k \subset \ldots \subset G$ all of which admit a Hamiltonian cycle, unless we subtract at least $n-1$ edges from $G$. The proof is a simple induction on $k$, the order of the subgraph.

First, observe that there are $n-1$ subgraphs of  order 2 in $G$ that contain $\{v_1\}$; they are the graphs induced by the vertex sets $\{v_1, v_j\}$ for $j \neq 1$. Observe that each of these graphs contain a Hamiltonian cycle of order $2$, namely $(1j)$, and that no edge appears more than once in these Hamiltonian cycles. Hence, if one subtracts at most $n-2$ edges from $G$ to obtain a graph $G'$, each edge subtraction breaking at most one cycle, there exists a subgraph of order $2$ in $G'$ which has a Hamiltonian decomposition. This completes the proof of the base case.

For the inductive step, consider a subgraph $G_k$ of order $k$ in $G$ and assume that $G_k$ has a Hamiltonian cycle. We show that unless we subtract at least $n-1$ edges from $G$ to obtain $G'$, there exists a subgraph $G_{k+1}$ of $G'$ of order $k+1$ which posses a Hamiltonian cycle. Assume without loss of generality that $(12\ldots k)$ is the Hamiltonian cycle in $G_k$ and let $j \in \lbrace k+1,\ldots n\}$. Then the subgraphs $G_{k+1}^j$  of $G$ induced by the vertex sets $\{1,2,\ldots,k,j\}$ have each $k$ Hamiltonian cycles: $(j12\ldots k), (1j2\ldots k), \ldots (12\ldots jk)$ as illustrated in Figure~\ref{fig:exproofthegap}. There are $n-k$ different subgraphs $G_{k+1}^j$ and, as above, no edge that is not in the original cycle $(12\ldots k)$ appears more than once. Hence, unless we subtract $k(n-k)$ edges from $G$ to obtain $G'$, there exists a graph of order $k+1$ in $G'$  which contains a Hamiltonian cycle. Because $k(n-k) \geq n-1$ for any $k \in \{1,2,\ldots, n-1\}$, this concludes the inductive step and the proof of point 3.

\end{proof}

{

}

\end{document}